\documentclass[11pt]{amsart}
\usepackage{graphicx}
\usepackage{amssymb}
\usepackage{amsmath}
\usepackage{amsthm,amsfonts,bbm}
\usepackage{amscd}
\usepackage{geometry}
\usepackage{epsfig,epstopdf}
\usepackage{mathrsfs}
\usepackage[backref]{hyperref}
\usepackage{cite}
\usepackage{color}

\newtheorem{thm}{Theorem}[section]
\newtheorem{cor}[thm]{Corollary}
\newtheorem{fact}[thm]{Fact}
\newtheorem{lem}[thm]{Lemma}
\theoremstyle{definition}
\newtheorem{defi}[thm]{Definition}
\theoremstyle{remark}
\newtheorem{rmk}[thm]{\bf Remark}

\numberwithin{equation}{section}
\numberwithin{figure}{section}
\def \Col{\textup{Col}}
\def \e{\varepsilon}
\def \ex{\text{ex}}
\def \EX{\text{EX}}
\def \spex{\text{spex}}

\begin{document}
\title[Spectral Tur\'an problems for nondegenerate hypergraphs]
{Spectral Tur\'an problems for nondegenerate hypergraphs}

\author[J. Zheng]{Jian Zheng}
\address{School of  Mathematics and Statistics, Jiangxi Normal University, Nanchang 330022,  China}
\email{zhengj@jxnu.edu.cn}

\author[H. Li]{Honghai Li$^\dag$}
\address{School of  Mathematics and Statistics, Jiangxi Normal University, Nanchang 330022,  China}
\email{lhh@jxnu.edu.cn}
\thanks{$^\dag$The corresponding author. H. Li was supported by National Natural Science Foundation of China (Nos. 12561060, 12161047, 12061038)}

\author[Y.-Z. Fan]{Yi-Zheng Fan$^\ddag$}
\address{School of Mathematical Sciences, Anhui University, Hefei 230601, China}
\email{fanyz@ahu.edu.cn}
\thanks{$^\ddag$Y.-Z. Fan was supported by National Natural Science Foundation of China (Nos. 12331012, 12471320).}

\keywords{Nondegenerate hypergraph; Tur\'an pair; $\alpha$-spectral radius; degree-stability; spectral stability}

\begin{abstract}
Keevash, Lenz and Mubayi developed a general criterion for hypergraph spectral extremal problems in their seminal work (SIAM J. Discrete Math., 2014). Their framework shows that extremal results on the $\alpha$-spectral radius (for $\alpha > 1$) may be deduced from a corresponding hypergraph Tur\'an problem exhibiting stability properties, provided its extremal construction satisfies certain continuity assumptions.
In this paper, we establish a spectral stability result for nondegenerate hypergraphs, extending the Keevash--Lenz--Mubayi criterion. Applying this result, we derive two general spectral Tur\'an theorems for hypergraphs with bipartite or multipartite pattern, thereby transforming  spectral Tur\'an problems into the corresponding purely combinatorial problems related to degree-stability in  nondegenerate $k$-graph families.
As  applications, we determine the maximum $\alpha$-spectral radius for several classes of hypergraph and characterize the corresponding extremal hypergraphs, such as the expansion of complete graphs, the generalized fans, the cancellative hypergraphs, the generalized triangles, and a special book hypergraph.
\end{abstract}

\maketitle

\section{Introduction}

A \emph{hypergraph} $H=(V(H),E(H))$ consists of a vertex set $V(H)=\{v_1,v_2,{\cdots},v_n\}$  and an edge set $E(H)=\{e_1,e_2,{\cdots},e_m\}$ , where $e_i \subseteq V$ for $i \in [m]:=\{1,2,\ldots,m\}$. If $|e_i|=k$ for each $i \in [m]$ and $k \geq2$, then $H$ is called a \emph{$k$-uniform} hypergraph (or simply \emph{$k$-graph}).
 A simple graph is exactly a $2$-uniform hypergraph. We denote by $e(H)$ the number of edges of $H$, that is, $e(H)=|E(H)|$.
 A $k$-graph $H'=(V(H'), E(H'))$ is called a \emph{sub-hypergraph} of  $H$ if $V(H')\subseteq V(H)$ and $E(H')\subseteq E(H)$.
 For $S\subseteq V(H)$,  the $k$-graph with $S$ as its vertex set and $E(H)\cap \binom{S}{k}$ as its edge set is called an \emph{induced sub-hypergraph} of $H$, denoted by $H[S]$. 

Let $\mathcal{F}$ be a family of $k$-graphs.
We say that a hypergraph $H$ is \emph{$\mathcal{F}$-free} if $H$ does not
contain any member of  $\mathcal{F}$ as a sub-hypergraph.
The \emph{Tur\'an number} $\ex(n,\mathcal{F})$ is defined as the maximum number of edges of an $\mathcal{F}$-free $k$-graph on  $n$ vertices.
Denote by $\EX(n,\mathcal{F})$ the set of all $\mathcal{F}$-free $k$-graphs with $ex(n,\mathcal{F})$ edges and $n$ vertices.
Determining the exact Tur\'an number for a general $k$-graph is a classic and intractable problem in extremal combinatorics, but if we are satisfied with the asymptotic results,  the  simple graph is completely solved (see \cite{ES1966}).
The  \emph{Tur\'an density} of $\mathcal{F}$ is defined as $$\pi(\mathcal{F})=\lim\limits_{n\to \infty}\frac{\ex(n,\mathcal{F})}{\binom{n}{k}},$$ and  $\mathcal{F}$ is called \emph{nondegenerate} if $\pi(\mathcal{F})>0$.
So, finding an asymptotic result for $\ex(n,\mathcal{F})$ is equivalent to determining the Tur\'an density if $\mathcal{F}$ is nondegenerate.

The Tur\'an problems are closely related to the phenomenon of stability, and many Tur\'an problems can be solved by the stability results of the corresponding graphs or hypergraphs. The first stability theorem was proved independently by Erd\H{o}s and Simonovits \cite{S1968}. In addition,  Simonovits \cite{S1968}  determined $\ex(n,F)$ exactly according to the stability theorem for a color-critical graph $F$.
At present, there is much  research on the stability of hypergraphs; for details see \cite{BIJ2017,FS2005,KS2005,MO2006,O2013,SL2018}.
In \cite{LRM2023}, Liu, Mubayi, and Reiher provided a unified framework for the stability of certain hypergraph families, which simplifies the proofs of many known results.

The spectral Tur\'an problem of graphs or hypergraphs is a spectral version of the Tur\'an problem.
Nikiforov made important contributions to the spectral Tur\'an problems of simple graphs.
For example,  Nikiforov \cite{N2007} determined the maximum spectral  radius for the $K_{l+1}$-free graph on $n$ vertices, and showed that the Tur\'an graph $T_{l}(n)$ is the unique spectral extremal graph, which is a generalization of Tur\'an theorem.
To date there are very few results on spectral Tur\'an problems of hypergraph.
In \cite{KLM2014}, Keevash, Lenz, and Mubayi gave two general criteria that formalize a generalized form of the strong stability of Tur\'an problems.
They also determined the maximum $\alpha$-spectral radius of any $3$-graph on $n$ vertices not containing the Fano plane when $n$ is sufficiently large.
In \cite{ELW2022}, Ellingham, Lu, and Wang characterized the extremal hypergraph with maximum spectral radius among all outerplanar $3$-graphs of $n$ vertices by its shadow.
In \cite{NLK2022}, Ni, Liu and Kang obtained the maximum $\alpha$-spectral radius of  cancellative $3$-graphs, and characterized the extremal hypergraph.
Hou, Liu, and Zhao \cite{HLZ2024} gave a result on spectral Tur\'an problems for some hypergraphs which has degree-stability.
Recently, the Tur\'an and spectral Tur\'an problems of linear hypergraphs have also been extensively studied; see \cite{FG2020,AC2021,GCH2022,HCC2021,SFK2025,SF2022}.

 In this paper, we investigate the spectral Tur\'an problems for  nondegenerate hypergraphs.  We show that 
for any family $\mathcal{F}$ of nondegenerate  $k$-graphs, under a certain growth condition on $\alpha$-spectral radius, 
the spectral Tur\'an  problems can be effectively reduced to the spectral extremal problems restricted to the class of $\mathcal{F}$-free hypergraphs with large  minimum degree. Furthermore, we give two general results for hypergraphs with bipartite or multipartite pattern, which  transform the corresponding spectral Tur\'an problems into purely combinatorial problems involving the degree-stability of  nondegenerate $k$-graph families; see Section $3$.
As an application,  we determine the maximum $\alpha$-spectral radius for some classes of hypergraphs and characterize the corresponding extremal hypergraphs, such as the expansion of complete graphs, the generalized fans, the cancellative hypergraphs, the generalized triangles, and a special book hypergraph; see Section $4$.

\section{Preliminaries}
\subsection{Stability}
Let $l$, $k$ be positive integers such that $l\geq k\geq2$.
A  $k$-graph is called \emph{$l$-partite} if its vertex set can be divided into $l$ parts, so that each edge contains at most one vertex from each part.
An edge maximal $l$-partite $k$-graph is called a \emph{complete $l$-partite $k$-graph}.
Let $T^{k}_{l}(n)$ be the balanced complete $l$-partite $k$-graph on $n$ vertices, namely, any two parts have sizes differing  by at most $1$.
Therefore, the number of edges in $T^{k}_{l}(n)$ is
\begin{equation*}
t^{k}_{l}(n):=\sum_{S\in \binom{[l]}{k}}\prod_{i\in S}n_{i},
\end{equation*}
where $n_{i}=\lfloor(n+i-1)/l\rfloor$ for $i\in [l]$.

Although $t^{k}_{l}(n)$ has an explicit expression, the following asymptotic result is more useful later in our estimation.

\begin{lem}\label{t}
Let $l\geq k\geq2$. Then 
$$t^{k}_{l}(n)=\frac{(l)_{k}}{k!l^{k}}n^{k}+O(n^{k-2}),$$ 
where $(l)_{k}:=l(l-1)\cdots (l-k+1)$.
\end{lem}

\begin{proof}
Let $n=lq+s$, where $0\leq s<l$. Then
\begin{displaymath}
\begin{split}
t^{k}_{l}(n)
=& \sum^{k}_{i=0}\binom{s}{i} \binom{l-s}{k-i}(q+1)^{i}q^{k-i}\\
=& \sum^{k}_{i=0}\binom{s}{i} \binom{l-s}{k-i}(q^{k}+iq^{k-1}+O(n^{k-2}))\\
=& \sum^{k}_{i=0}\binom{s}{i} \binom{l-s}{k-i}\big(\frac{n}{l}-\frac{s}{l}\big)^{k}+\sum^{k}_{i=0}i \binom{s}{i} \binom{l-s}{k-i}\big(\frac{n}{l}-\frac{s}{l})^{k-1}+O(n^{k-2}\big).
\end{split}
\end{displaymath}
So we have
\begin{displaymath}
\begin{split}
t^{k}_{l}(n)
=& \binom{l}{k} \big(\frac{n}{l}-\frac{s}{l}\big)^{k}+\sum^{k}_{i=1}s \binom{s-1}{i-1} \binom{l-s}{k-i}\big(\frac{n}{l}-\frac{s}{l})^{k-1}+O(n^{k-2}\big)\\
=& \frac{(l)_{k}}{k!l^{k}}n^{k}-\frac{ks}{l^{k}}\binom{l}{k}n^{k-1}+s \sum^{k-1}_{i=0}\binom{s-1}{i} \binom{l-s}{k-1-i}\big(\frac{n}{l})^{k-1}+O(n^{k-2}\big)\\
=& \frac{(l)_{k}}{k!l^{k}}n^{k}-\frac{ks}{l^{k}}\binom{l}{k}n^{k-1}+\frac{s}{l^{k-1}}\binom{l-1}{k-1}n^{k-1}+O(n^{k-2})\\
=& \frac{(l)_{k}}{k!l^{k}}n^{k}+O(n^{k-2}).
\end{split}
\end{displaymath}
\end{proof}

 A \emph{$k$-multiset} is a  collection of $k$ elements with repetitions allowed.
 A \emph{$k$-pattern} is a pair $P=([l],E)$ where $l$ is a positive integer and $E$ is a collection of $k$-multisets with elements from $[l]$.
 Clearly, $k$-patterns are generalizations of $k$-graphs.
 Given a $k$-graph $H$ and $k$-pattern $P=([l],E)$, a map $\phi$: $V(H)\rightarrow [l]$ is a \emph{homomorphism} from $H$ to $P$ if the $k$-multiset $\{\phi(v_{1}),\ldots,\phi(v_{k})\}$ belongs to $E$ for every edge $\{v_{1},\ldots,v_{k}\}\in E(H)$.
 We say $H$ is \emph{$P$-colorable} if there is a homomorphism from $H$ to $P$.
 For example, any $l$-partite $k$-graph is $K^{k}_{l}$-colorable, where $K^{k}_{l}$ is the complete $k$-graph on $l$ vertices; conversely, every $K^{k}_{l}$-colorable $k$-graph is $l$-partite.
 Let $\mathcal{F}$ be a family of $k$-graphs and $P$ be a $k$-pattern.
 We say $(\mathcal{F},P)$ is a \emph{Tur\'an pair} if every $P$-colorable $k$-graph is $\mathcal{F}$-free and every edge maximum $\mathcal{F}$-free $k$-graph is $P$-colorable.

For a $k$-graph $H$ and a vertex $v\in V(H)$,  we denote by $E_{H}(v)$  the set of edges in $H$ containing the vertex $v$.
The degree $d_{H}(v)$ of $v$ is defined as the cardinality $|E_{H}(v)|$. Let $\delta(H)$ denote the minimum degree of $H$. For simplicity,
we write $H-v$ for the induced sub-hypergraph $H[V(H)\backslash\{v\}]$.

\begin{defi}[\hspace{1sp}\cite{HLZ2024}]
Let $\mathcal{F}$ be a family of nondegenerate $k$-graphs, where $k\geq 2$,  and let $\mathfrak{H}$ be a family of $\mathcal{F}$-free $k$-graphs.

\begin{itemize}
\item[(1)] $\mathcal{F}$ is \emph{edge-stable} with respect  to $\mathfrak{H}$ if for every $\delta>0$ there exist $\e>0$ and  $n_{0}$ such that  every $\mathcal{F}$-free $k$-graph $\mathcal{H}$ on $n\geq n_{0}$ vertices with $e(\mathcal{H})\geq(\pi(\mathcal{F})/k!-\e)n^{k}$  becomes a member of $\mathfrak{H}$ after removing at most $\delta n^{k}$ edges.

\item[(2)] $\mathcal{F}$ is \emph{degree-stable} with respect  to $\mathfrak{H}$  if  there exist $\e>0$ and $n_{0}$ such that every $\mathcal{F}$ free $k$ -graph $\mathcal{H}$ on the $n\geq n_{0}$ vertices with $\delta(\mathcal{H})\geq (\pi(\mathcal{F})/(k-1)!-\e)n^{k-1}$ is a member of $\mathfrak{H}$.

\item[(3)] $\mathcal{F}$ is \emph{vertex-extendable} with respect  to $\mathfrak{H}$  if  there exist  $\e>0$ and $n_{0}$ such that every $\mathcal{F}$ free $k$-graph $\mathcal{H}$ on $n\geq n_{0}$ vertices with $\delta(\mathcal{H})\geq (\pi(\mathcal{F})/(k-1)!-\e)n^{k-1}$  satisfies: if $\mathcal{H}-v$ is a member of $\mathfrak{H}$ for some vertex $v$, then $\mathcal{H}$ is a member of $\mathfrak{H}$ as well.
\end{itemize}
\end{defi}

 A class $\mathfrak{H}$ of $k$-graphs is called \emph{hereditary} if it is closed under taking induced sub-hypergraphs, that is, for every  $G \in \mathfrak{H}$ and $S\subseteq V(G)$, we have $G[S]\in\mathfrak{H}$.
 Note that the collection of all $P$-colorable hypergraphs is hereditary.
 In many cases, the extremal hypergraphs of Tur\'an problems are $P$-colorable for some pattern $P$, so we usually choose $\mathfrak{H}$ as the collection of all $P$-colorable hypergraphs.
 For further developments on the hereditary property of hypergraphs, see \cite{N2014A, N2014S}.

\begin{thm}[\hspace{1sp}\cite{HLZ2024}]\label{edv}
Let $\mathcal{F}$ be a family of nondegenerate $k$-graphs and $\mathfrak{H}$ be a hereditary class of $\mathcal{F}$-free $k$-graphs.
If $\mathcal{F}$ is both edge-stable and vertex-extendable with respect to $\mathfrak{H}$, then $\mathcal{F}$ is degree-stable with respect to $\mathfrak{H}$.
\end{thm}

\subsection{Spectral radius}
Let $G$ be a $k$-graph on $n$ vertices. For any $\alpha>1$, the \emph{Lagrangian polynomial} $P_{G}(\mathbf{x})$ of $G$ is defined
 as $$P_{G}(\mathbf{x})=k!\sum_{\{i_{1},\ldots,i_{k}\}\in E(G)}x_{i_{1}}\cdots x_{i_{k}}.$$
 Let $\mathbf{x}=(x_{1},\ldots,x_{n})\in \mathbb{R}^{n}$.
 For $\alpha \ge 1$, the $\ell_\alpha$-norm of $\mathbf{x}$ is denoted and defined by 
  $\|\mathbf{x}\|_{\alpha}:=(|x_{1}|^{\alpha}+\cdots +|x_{n}|^{\alpha})^{1/\alpha}$.
The \emph{$\alpha$-spectral radius} of $G$ is defined as
$$\lambda_{\alpha}(G)=\max_{\|\mathbf{x}\|_{\alpha}=1}P_{G}(\mathbf{x}).$$
If taking $\mathbf{x}$ be an all-one vector, then 
\begin{equation}\label{ave}
    \lambda_\alpha(G) \ge k! n^{-k/\alpha} e(G).
\end{equation}
If $\mathbf{x}\in \mathbb{R}^{n}$ is a unit vector with respect to $\ell_\alpha$-norm such that $\lambda_{\alpha}(G)=P_{G}(\mathbf{x})$, then $\mathbf{x}$ is called an \emph{eigenvector} of $G$ corresponding to $\lambda_{\alpha}(G)$.
 Obviously, any $k$-graph $G$  always has a nonnegative eigenvector corresponding to $\lambda_{\alpha}(G)$.
When $\alpha>1$, the nonnegative eigenvector $\mathbf{x}=(x_{1},\ldots,x_{n})$ of a
 $k$-graph $G$ satisfies the following equations derived from Lagrange's method:
\begin{equation*}
\lambda_{\alpha}(G)x_{i}^{\alpha-1}=(k-1)!\sum_{\{i,i_{2},\ldots,i_{k}\}\in E(G)}x_{i_{2}}\cdots x_{i_{k}},~ \mbox{for}~
1\leq i\leq n.
\end{equation*}

The \emph{spectral Tur\'an number}, denoted by $\spex(n,\mathcal{F})$, is defined as the maximum $\alpha$-spectral radius over all $\mathcal{F}$-free $k$-graphs on $n$ vertices. The \emph{spectral Tur\'an density}  can be analogously defined  by 
$$\lambda^{(\alpha)}(\mathcal{F}):=\lim\limits_{n\to \infty}\frac{\spex(n,\mathcal{F})}{n^{k-k/\alpha}}.$$

Note that any family of $\mathcal{F}$-free $k$-graphs is hereditary. The result of Nikiforov \cite[Theorem $9.3$]{N2014A} (see also \cite[Theorem 12]{N2014S}) implies the following:

 \begin{lem}[\hspace{1sp}\cite{N2014A,N2014S}]\label{sed}
Let $\mathcal{F}$ be a family of $k$-graphs. Then for every $\alpha>1$,
$$\lambda^{(\alpha)}(\mathcal{F})=\pi(\mathcal{F}).$$
\end{lem}

\begin{lem}[\hspace{1sp}\cite{KNY2014}]\label{aut}
Let $G$ be a  $k$-graph of order $n$ with at least one edge, and let $u$ and $v$ be vertices of $G$ such that the transposition of $u$ and $v$ is an automorphism of $G$.
If $\alpha>1$, and $\mathbf{x}$ is an eigenvector corresponding to $\lambda_{\alpha}(G)$, then $x_{u}=x_{v}$.
\end{lem}

\begin{lem}[\hspace{1sp}\cite{KNY2014}]\label{nonne}
Let $\alpha\geq1$, and let $G$ be a $k$-graph such that every nonnegative eigenvector corresponding to $\lambda_{\alpha}(G)$ is positive.
If $H$ is a sub-hypergraph of $G$, then $\lambda_{\alpha}(H)<\lambda_{\alpha}(G)$, unless $H=G$.
\end{lem}

\begin{thm}[\hspace{1sp}\cite{KNY2014}]\label{kpt}
Let $l\geq k\geq2$, and let $G$ be an $l$-partite $k$-graph of order $n$.
For every $\alpha>1$,
$$\lambda_{\alpha}(G)\leq \lambda_{\alpha}(T^{k}_{l}(n)),$$
with equality if and only if $G=T^{k}_{l}(n)$.
\end{thm}

\begin{thm}[\hspace{1sp}\cite{KNY2014}]\label{kpg}
Let $l\geq k\geq2$, and let $G$ be an $l$-partite $k$-graph of order $n$.
If $\alpha>1$, then
$$\lambda_{\alpha}(G)\leq \frac{(l)_{k}}{l^{k}}n^{k-k/\alpha},$$
with equality if and only if $l\mid n$ and $G=T^{k}_{l}(n)$.
\end{thm}

\section{Main Results}

In this section, we give a spectral stability  result for nondegenerate hypergraphs. Applying the result, we present two general theorems for  determining the maximum $\alpha$-spectral radius among all $n$-vertex $\mathcal{F}$-free $k$-graphs, where $\mathcal{F}$ is a certain family of non-degenerate $k$-graphs.

\begin{thm}[Spectral stability]\label{cri}
Let $k\geq2$, $\alpha>1$, $0<\e<1$, and $\mathcal{F}$ be a family of $k$-graphs with $\pi(\mathcal{F})>0$. Let $\mathcal{G}_{n}$ be the collection of all $n$-vertex $\mathcal{F}$-free $k$-graphs with minimum degree at least  $(1-\e)\pi(\mathcal{F})\binom{n}{k-1}$, and let
$\lambda_{\alpha}(\mathcal{G}_{n})=\max\{\lambda_{\alpha}(G):G\in \mathcal{G}_{n}\}$.
Suppose that
there exists  $N$ such that
 for all $n\geq N$,
\begin{equation}\label{cri1}
\lambda_{\alpha}(\mathcal{G}_{n}) \ge  \lambda_{\alpha}(\mathcal{G}_{n-1}) + (k-k/\alpha)(1-\e')\pi(\mathcal{F})n^{k-k/\alpha-1},
\end{equation}
where $\e'=\e\pi(\mathcal{F})(\alpha-1)/(2k \alpha)$.
Then there exists $n_{0}$ such that if $H$ is an $\mathcal{F}$-free graph on $n\geq n_{0}$ vertices, then
$$\lambda_{\alpha}(H)\leq \lambda_{\alpha}(\mathcal{G}_{n}).$$
In addition, if the equality holds, then $H\in\mathcal{G}_{n}$.
\end{thm}

\begin{rmk}
The proof of Theorem \ref{cri} is  lengthy, so we defer it to Section \ref{Sec-5}. Furthermore,  for all sufficiently large $n$, if 
\begin{equation}\label{KLM-1}
\bigg|\ex(n,\mathcal{F})-\ex(n-1,\mathcal{F})-\pi(\mathcal{F})\binom{n}{k-1}\bigg|<c n^{k-1}
\end{equation}
and
\begin{equation}\label{KLM-2}
\bigg|\lambda_{\alpha}(\mathcal{G}_{n})-k!\ex(n,\mathcal{F})n^{-k/\alpha}\bigg|\leq c n^{k-k/\alpha-1},
\end{equation}
then there exists a constant $r$ such that
$$\lambda_{\alpha}(\mathcal{G}_{n}) \ge  \lambda_{\alpha}(\mathcal{G}_{n-1}) + (k-k/\alpha)(1-rc)\pi(\mathcal{F})n^{k-k/\alpha-1}.$$
 Therefore, under the conditions \eqref{KLM-1} and \eqref{KLM-2}, for sufficiently small $c$, by Theorem \ref{cri}, we also get the corresponding result, which generalizes the criterion by Keevash, Lenz and Mubayi \cite[Theorem $1.4$]{KLM2014}.
\end{rmk}

For a $k$-pattern $P$, we use $\Col(P)$ to denote the set of all
$P$-colorable $k$-graphs. 

\begin{thm}\label{PL}
Let $(\mathcal{F},K^{k}_{l})$ be a Tur\'an pair, where $\mathcal{F}$ is a family of $k$-graphs that is degree-stable with respect to $\Col(K^{k}_{l})$, and $l\geq k\geq 2$.
For any $\alpha>1$, there exists $n_{0}$ such that if $G$ is an  $\mathcal{F}$-free $k$-graph on $n \ge n_{0}$ vertices, then $\lambda_{\alpha}(G)\leq\lambda_{\alpha}(T^{k}_{l}(n))$, with equality if and only if $G=T^{k}_{l}(n)$.
\end{thm}

\begin{proof}
Since $(\mathcal{F},K^{k}_{l})$ is a Tur\'an pair, by Lemma \ref{t} we have
$$\ex(n,\mathcal{F})=e(T^{k}_{l}(n))=t^{k}_{l}(n),~\pi(\mathcal{F})=\lim\limits_{n\to \infty}\frac{t^{k}_{l}(n)}{\binom{n}{k}}=\frac{(l)_{k}}{l^{k}}.$$
By the degree-stability of $\mathcal{F}$ with respect to $\Col(K^{k}_{l})$,  there exist $n_{0}$ and $\e' \in(0,1)$ such that   every $\mathcal{F}$-free $k$-graph $H$ on $n\geq n_{0}$ vertices with $\delta(H)\geq (\pi(\mathcal{F})/(k-1)!-\e')n^{k-1}$ is contained in $\Col(K^{k}_{l})$.
Let $\e:=\e'/2$ and define $\mathcal{G}_n$ as the family  of all $n$-vertex $\mathcal{F}$-free $k$-graphs $G$ with  $\delta(G) \ge (1-\e)\pi(\mathcal{F})\binom{n}{k-1}$.  
For sufficiently large $n$ and any $G \in \mathcal{G}_n$, we have
$$\delta(G) \ge (1-\e)\pi(\mathcal{F})\binom{n}{k-1}\geq (\pi(\mathcal{F})/(k-1)!-\e')n^{k-1},$$
which implies that $G$ is a member of $\Col(K_l^k)$ and hence $G$ is $l$-partite.

By the definition of Tur\'an pair, $T^{k}_{l}(n)$ is $\mathcal{F}$-free as it is $K_l^k$-colorable, and it satisfies
$$\delta(T^{k}_{l}(n))\geq e(T^{k}_{l}(n))-e(T^{k}_{l}(n-1))=\frac{\pi(\mathcal{F})}{(k-1)!}n^{k-1}+O(n^{k-2}),$$
implying that $T^{k}_{l}(n)\in \mathcal{G}_n$.
By  Lemma \ref{t} and Theorem \ref{kpg}, for sufficiently large $n$,
\begin{displaymath}
\begin{split}
\lambda_{\alpha}(\mathcal{G}_{n})-\lambda_{\alpha}(\mathcal{G}_{n-1})
&\geq \lambda_{\alpha}(T^{k}_{l}(n))-\frac{(l)_{k}}{l^{k}}(n-1)^{k-k/\alpha}\\
&\geq k!e(T^{k}_{l}(n))n^{-k/\alpha}-\frac{(l)_{k}}{l^{k}}(n-1)^{k-k/\alpha}\\
&= \frac{(l)_{k}}{l^{k}}n^{k-k/\alpha}+O(n^{k-k/\alpha-2})-\frac{(l)_{k}}{l^{k}}(n-1)^{k-k/\alpha}\\
&=(k-k/\alpha)\pi(\mathcal{F})n^{k-k/\alpha-1}+o(n^{k-k/\alpha-1}),
\end{split}
\end{displaymath}
where the 2nd inequality follows from \eqref{ave}.
So, we get the result by Theorems \ref{cri} and \ref{kpt}.
\end{proof}

A $2k$-graph $G$ is called \emph{bipartite-like} if its vertex set has a bipartition such that each edge contains exactly $k$ vertices from each part.
An edge maximal bipartite-like $2k$-graph is called \emph{complete bipartite-like}.
Let $B_{2k}(n)$ be the complete balanced bipartite-like $2k$-graph on $n$ vertices.

\begin{lem}\label{new}
Let $\alpha>1$, and let $G$ be a complete bipartite-like $4$-graph on $n$ vertices.
If $H$ is a sub-hypergraph of $G$, then $\lambda_{\alpha}(H)<\lambda_{\alpha}(G)$, unless $H=G$.
\end{lem}
\begin{proof}
Let $V_{1}$ and $V_{2}$ be the two parts $G$, and let $\mathbf{x}$ be a nonnegative eigenvector of $G$ corresponding to $\lambda_{\alpha}(G)$.
From Lemma \ref{aut},  all entries of $\mathbf{x}$ indexed by vertices in the same part are equal.
If there exists a vertex $u\in V_{j}$ ($j\in \{1,2\}$) such that $x_{u}=0$, then $x_{v}=0$ for all $v\in V_{j}$. 
This implies $\lambda_{\alpha}(G)=0$, which is impossible since $G$ is nonempty.
Thus, every nonnegative eigenvector corresponding  to $\lambda_{\alpha}(G)$ is positive.
The result follows from Lemma \ref{nonne}.
\end{proof}

\begin{lem}\label{B4}
Let $G$ be a bipartite-like $4$-graph on $n$ vertices.
If $\alpha>1$, then
$$\lambda_{\alpha}(G)\leq \lambda_{\alpha}(B_{4}(n))\leq \frac{3}{8}(n-2)^{2}n^{2-4/\alpha},$$
with left equality if and only if $G=B_{4}(n)$ and right equality if and only if $2 \mid n$.
\end{lem}

\begin{proof}
Suppose that $G$ is a bipartite-like $4$-graph on $n$ vertices with maximum $\alpha$-spectral radius.
By Lemma \ref{new}, $G$ is a complete bipartite-like $4$-graph with a positive eigenvector $\mathbf{x}$ corresponding to $\lambda_{\alpha}(G)$.
Let $V_{1}$ and $V_{2}$ be two parts of $G$, where $|V_{1}|:=t$.
By Lemma \ref{aut}, 
 we can assume that $x_{v}=:(\frac{\gamma}{t})^{1/\alpha}$ if $v\in V_{1}$ and $x_{v}=:(\frac{1-\gamma}{n-t})^{1/\alpha}$ if $v\in V_{2}$, for some $ \gamma \in (0,1)$.
 Then
 \begin{displaymath}
\begin{split}
 \lambda_{\alpha}(G)&=
 4!\max_{0< \gamma < 1}\binom{t}{2}\binom{n-t}{2}\Big(\frac{\gamma}{t}\Big)^{2/\alpha}\Big(\frac{1-\gamma}{n-t}\Big)^{2/\alpha}\\
 &=4!\times2^{-4/\alpha}\binom{t}{2}\binom{n-t}{2}t^{-2/\alpha}(n-t)^{-2/\alpha}.
\end{split}
\end{displaymath}

Consider the following function of $t$ on $[2,n-2]$:
\begin{displaymath}
\begin{split}
f(t)&=\binom{t}{2}\binom{n-t}{2}t^{-2/\alpha}(n-t)^{-2/\alpha}\\
&=\frac{1}{4}t^{1-2/\alpha}(n-t)^{1-2/\alpha}(t-1)(n-t-1).
\end{split}
\end{displaymath}
Write  $g_{1}(t)=\frac{1}{4}t^{1-2/\alpha}(n-t)^{1-2/\alpha}$ and $h_{1}(t)=(t-1)(n-t-1)$.
If $\alpha\geq2$, noting that $g_{1}(t)$ and $h_{1}(t)$ are nonnegative, symmetric with respect to $t=n/2$, increasing in $[2,n/2]$ and decreasing in $[n/2,n-2]$, we have $f(n/2)=\max\{f(t): t\in [2, n-2]\}$.
Observe that
$$f(t)=\frac{1}{4}t^{2-2/\alpha}(n-t)^{2-2/\alpha}+\frac{1}{4}(1-n)t^{1-2/\alpha}(n-t)^{1-2/\alpha}.$$
Write $g_{2}(t)=\frac{1}{4}t^{2-2/\alpha}(n-t)^{2-2/\alpha}$ and $h_{2}(t)=\frac{1}{4}(1-n)t^{1-2/\alpha}(n-t)^{1-2/\alpha}$.
If $1<\alpha<2$,  noting that $g_{2}(t)$ and $h_{2}(t)$ are symmetric with respect to $t=n/2$,
 increasing on $[2,n/2]$ and decreasing on $[n/2,n-2]$,
we still have $f(n/2)=\max\{f(t): t\in [2, n-2]\}$.

In summary, for $\alpha>1$,
$$\lambda_{\alpha}(G)\leq 4!\times2^{-4/\alpha}f(\lfloor n/2\rfloor)\leq  4!\times2^{-4/\alpha}f(n/2),$$
or equivalently,
$$\lambda_{\alpha}(G)\leq \lambda_{\alpha}(B_{4}(n))\leq \frac{3}{8}(n-2)^{2}n^{2-4/\alpha}.$$
The result follows.
\end{proof}

Now let us focus on the pattern $P=(\{1,2\},\{\{1,1,2,2\}\})$.
Note that a $4$-graph is $P$-colorable is and only if it is  bipartite-like. Moreover, it is straightforward to verify that $B_4(n)$ attains the maximum number of edges among all bipartite-like $4$-graphs on $n$ vertices. By simple calculation, we have
\begin{equation}\label{B4n}
e(B_4(n))=\binom{\lfloor n/2\rfloor}{2}\binom{\lceil n/2\rceil}{2}=\frac{n^{4}-4n^{3}}{64}+O(n^{2}).
\end{equation}

\begin{thm}\label{PB4}
Let $(\mathcal{F},P)$ be a  Tur\'an pair with $P=(\{1,2\},\{\{1,1,2,2\}\})$, where $\mathcal{F}$ is a family of $4$-graphs that is degree-stable with respect to $\Col(P)$.
For any $\alpha>1$, there exists $n_{0}$ such that if $G$ is an $\mathcal{F}$-free $4$-graph $G$ on $n>n_{0}$ vertices, then $\lambda_{\alpha}(G)\leq\lambda_{\alpha}(B_{4}(n))$, with equality if and only if $G=B_{4}(n)$.
\end{thm}

\begin{proof}
Since $(\mathcal{F},P)$ is a Tur\'an pair, by (\ref{B4n}) we have
$$\ex(n,\mathcal{F})=e(B_{4}(n))=\frac{n^{4}-4n^{3}}{64}+O(n^{2}),~\pi(\mathcal{F})=\lim\limits_{n\to \infty}\frac{e(B_{4}(n))}{\binom{n}{4}}=\frac{3}{8}.$$
By the degree-stability of $\mathcal{F}$ with respect to $\Col(P)$,  there exist $n_{0}$ and $\e \in(0,1)$ such that every $\mathcal{F}$-free $4$-graph $\mathcal{H}$ on $n\geq n_{0}$ vertices with $\delta(\mathcal{H})\geq (1/16-\e)n^{3}$ is contained in $\Col(P)$.
Define $\mathcal{G}_n$ as the family  of all $n$-vertex $\mathcal{F}$-free $4$-graphs $G$ with  $\delta(G) \ge \frac{3(1-\e)}{8}\binom{n}{3}$.  
For sufficiently large $n$ and any $G\in \mathcal{G}_n$, we have
$$\delta(G) \ge \frac{3(1-\e)}{8}\binom{n}{3}\geq \Big(\frac{1}{16}-\e\Big)n^{3},$$
which implies that $G$ is a member of $\Col(P)$ and hence $G$ is a bipartite-like $4$-graph.

By definition $B_{4}(n)$ is $\mathcal{F}$-free as it is $P$-colorable, and it holds 
$$\delta(B_{4}(n))\geq e(B_{4}(n))-e(B_{4}(n-1))=\frac{1}{16}n^{3}+O(n^{2}).$$
This implies $B_{4}(n)\in \mathcal{G}_n$.
By Lemma \ref{B4}, for sufficiently large $n$, we have
\begin{displaymath}
\begin{split}
\lambda_{\alpha}(\mathcal{G}_{n})-\lambda_{\alpha}(\mathcal{G}_{n-1})
\geq& \lambda_{\alpha}(B_{4}(n))-\frac{3}{8}(n-3)^{2}(n-1)^{2-4/\alpha}\\
\geq& 4!e(B_{4}(n))n^{-4/\alpha}-\frac{3}{8}\big(n^{4-4\alpha}-(8-4/\alpha)n^{3-4/\alpha}\big)+O(n^{2-4/\alpha})\\
\geq& \frac{3}{8}\big(n^{4-4/\alpha}-4n^{3-4/\alpha}\big)-\frac{3}{8}\big(n^{4-4\alpha}-(8-4/\alpha)n^{3-4/\alpha}\big)+O(n^{2-4/\alpha})\\
=& (4-4/\alpha)\pi(\mathcal{F})n^{3-4/\alpha}+o(n^{3-4/\alpha}).
\end{split}
\end{displaymath}
 The result now follows from Theorem \ref{cri} and Lemma \ref{B4}.
\end{proof}

\section{Applications}
In this section, we apply the results in Section $3$ to the spectral Tur\'an problems of $\mathcal{F}$-free hypergraphs for some special families $\mathcal{F}$.

\subsection{The expansion of complete graphs}
For a simple graph $G$ and an integer $k\geq 3$, the \emph{$k$-expansion of   $G$}, denoted by $G^{(k)}$, is the $k$-graph obtained from $G$ by enlarging each edge of $G$ with $k-2$ new vertices disjoint from $V(G)$ such that distinct edges of $G$ are enlarged by distinct vertices.
  In \cite{O2013}, Pikhurko proved that  $\EX(n,K^{(k)}_{l+1})=\{T^{k}_{l}(n)\}$ for any $l\geq k\geq3$ when $n$ is sufficiently large. So $(K^{(k)}_{l+1},K^{k}_{l})$ is a Tur\'an pair. Pikhurko also proved that $K^{(k)}_{l+1}$ is edge-stable with respect to  $\Col(K^{k}_{l})$ (see \cite[Lemma 3]{O2013}), and is also vertex-extendable (see Page 12 in \cite{HLL2023}).
  By Theorem \ref{edv}, we know that $K^{(k)}_{l+1}$ is degree-stable with respect to  $\Col(K^{k}_{l})$.
  Hence, by Theorem \ref{PL}, we obtain the following corollary.
  
\begin{cor}
For any $l\geq k\geq 3$ and $\alpha>1$, there exists $n_{0}$ such that if $G$ is a $K^{(k)}_{l+1}$-free $k$-graph on $n \ge n_{0}$ vertices, then $\lambda_{\alpha}(G)\leq\lambda_{\alpha}(T^{k}_{l}(n))$, with equality if and only if $G=T^{k}_{l}(n)$.
\end{cor}

\subsection{The  extension  of hypergraphs}
Let $k \ge 3$ and $F$ be a $k$-graph.
The \emph{extension} of $F$, denoted by $H^{F}$, is the $k$-graph obtained from $F$ by adding a $(k-2)$-set $B_{ij}$ of new vertices and an edge $\{v_{i},v_{j}\} \cup B_{ij}$ for each pair of vertices $v_i,v_j$ of $F$ that is not contained in any edge of $F$, and moreover, these $(k-2)$-sets of new vertices are pairwise disjoint.

The \emph{generalized $k$-fan}, denoted by $\text{Fan}^{k}$, is the extension of the $k$-graph on $k+1$ vertices with only one edge.
In \cite{MO2006}, Mubayi and Pikhurko proved that $\EX(n,\text{Fan}^{k})=\{T^{k}_{k}(n)\}$ for $k\geq3$ and sufficiently large $n$, and proved that Fan$^{k}$ is edge-stable with respect to $\Col(K^{k}_{k})$. 
So $(\text{Fan}^{k}, K^{k}_{k})$ is a Tur\'an pair.
It follows from Page $13$ in \cite{HLL2023} that $\text{Fan}^{k}$ is vertex-extendable.
So, by Theorem \ref{edv}, we know that $\text{Fan}^{k}$ is degree-stable with respect to $\Col(K^{k}_{k})$.
Therefore, by Theorem \ref{PL}, we obtain the following result.

\begin{cor}
For any $l\geq k\geq 3$ and $\alpha>1$, there exists $n_{0}$ such that if $G$ is a $\textup{Fan}^{k}$-free $k$-graph on $n \ge n_{0}$ vertices, then
$\lambda_{\alpha}(G)\leq\lambda_{\alpha}(T^{k}_{k}(n))$, with equality if and only if $G=T^{k}_{k}(n)$.
\end{cor}

 Let $M^{k}_{t}$  denote the $k$-graph consisting of $t$ vertex-disjoint edges, also called a \emph{$t$-matching}. 
 Let $S_{t}^{(k)}$ denote the $k$-expansion of a star $S_{t}$ with $t$ edges,  which is also called a \emph{$t$-hyperstar}.
 By Corollary $1.13$ and Concluding Remarks in \cite{LRM2023},  we see that the extension $H^{M^{3}_{t}}$ (respectively, $H^{S^{(3)}_{t}}$, $H^{S^{(4)}_{t}}$)
  is degree-stable with respect to  $\Col(K_{3t-1}^{3})$(respectively,  $\Col(K_{2t}^{3})$, $\Col(K_{3t}^{4})$) for $t\geq 2$.
  By Theorems 5.1-5.3 in \cite{JPW2018},  for $t\geq2$ and sufficiently large $n$,
  $$\EX(n, H^{M^{3}_{t}})=\{T^{3}_{3t-1}(n)\},~\EX(n, H^{S^{(3)}_{t}})=\{T_{2t}^{3}(n)\},~\EX(n, H^{S^{(4)}_{t}})=\{T_{3t}^{4}(n)\}.$$
  So $(H^{M^{3}_{t}},K^{3}_{3t-1})$, $(H^{S^{(3)}_{t}},K_{2t}^{3})$
   and $(H^{S^{(4)}_{t}},K_{3t}^{4})$ are all Tur\'an pairs.
   By Theorem \ref{PL}, we get the following spectral extremal results immediately.

\begin{cor}
For any $\alpha>1$ and $t\geq 2$, there exists $n_{0}$ such that if $G$ is a $H^{M^{3}_{t}}$-free $3$-graph on $n \ge n_{0}$ vertices, then
$\lambda_{\alpha}(G)\leq\lambda_{\alpha}(T^{3}_{3t-1}(n))$, with equality if and only if $G=T^{3}_{3t-1}(n)$.
\end{cor}

\begin{cor}
For any $\alpha>1$ and $t\geq 2$, there exists $n_{0}$ such that if $G$ is a $H^{S^{(3)}_{t}}$-free $3$-graph on $n \ge n_{0}$ vertices, then
$\lambda_{\alpha}(G)\leq\lambda_{\alpha}(T^{3}_{2t}(n))$, with equality if and only if $G=T^{3}_{2t}(n)$.
\end{cor}

\begin{cor}
For any $\alpha>1$ and $t\geq 2$, there exists $n_{0}$ such that if $G$ is a $H^{S^{(4)}_{t}}$-free $4$-graph on $n \ge n_{0}$ vertices, then
$\lambda_{\alpha}(G)\leq\lambda_{\alpha}(T^{4}_{3t}(n))$, with equality if and only if $G=T^{4}_{3t}(n)$.
\end{cor}

\subsection{Cancellative hypergraphs and generalized triangles}
A $k$-graph $G$ is called \emph{cancellative} if whenever $A$, $B$, $C$ are edges of $G$ with $A\cup B=A\cup C$ (or equivalently, $B \Delta C\subseteq A$, where $\Delta$ denotes the symmetric difference), we have $B=C$. In particular, a graph $G$ is cancellative if and only if it is triangle-free.

The Tur\'an problems of cancellative hypergraphs are closely related to the generalized triangles.
The \emph{generalized triangle} $T_{k}$ is the $k$-graph with vertex set $[2k-1]$ and edge set
$$\{\{1,\ldots,k-1,k\},\{1,\ldots,k-1,k+1\},\{k,k+1,\ldots,2k-1\}\}.$$
Note that a $3$-graph is cancellative if and only if it is $\{F_{4},T_{3}\}$-free, where $F_{4}$ is a hypergraph with edge set $\{\{1,2,3\},\{1,2,4\},\{1,3,4\}\}$.
Bollob\'as \cite{B} showed that $\EX(n,\{F_{4},T_{3}\})=\{T^{3}_{3}(n)\}$. Subsequently, Frankl and F\"{u}redi \cite{FF1983} proved that $\EX(n,T_{3})=\{T^{3}_{3}(n)\}$ for all $n\geq3000$, and this was improved to $n\geq33$ by Keevash and Mubayi \cite{KM2004}.
In \cite{O2008}, Pikhurko proved that $\EX(n,T_{4})=\{T^{4}_{4}(n)\}$ for sufficiently large $n$. 
So $(T_{3},K^{3}_{3})$ and  $(T_{4},K^{4}_{4})$ are both Tur\'an pairs. 
From Theorem 1.10 in \cite{LRM2023}, we know that
 $T_{3}$ (respectively, $T_{4}$) is degree-stable with respect to  $\Col(K^{3}_{3})$ (respectively, $\Col(K^{4}_{4})$).
 Therefore, by Theorem \ref{PL}, we have the following result.

\begin{cor}
For any $\alpha>1$ and $k\in \{3,4\}$, there exists $n_{0}$ such that if $G$ is a $T_{k}$-free $k$-graph on $n \ge n_{0}$ vertices, then
$\lambda_{\alpha}(G)\leq\lambda_{\alpha}(T^{k}_{k}(n))$, with equality if and only if $G=T^{k}_{k}(n)$.
\end{cor}

Since cancellative $k$-graphs must be $T_{k}$-free, we immediately obtain the following corollary, which implies the result in \cite[Theorem $3.2$]{NLK2022}.

\begin{cor}
For any $\alpha>1$ and $k\in \{3,4\}$, there exists $n_{0}$ such that if $G$ is a cancellative $k$-graph on $n \ge n_{0}$ vertices, then 
$\lambda_{\alpha}(G)\leq\lambda_{\alpha}(T^{k}_{k}(n))$, with equality if and only if $G=T^{k}_{k}(n)$.
\end{cor}

\subsection{4-book with three pages}
Let $F_{7}$ denote the $4$-graph with vertex set $\{1,2,\ldots,7\}$ and edge set
$\{\{1234\},\{1235\},\{1236\},\{4567\}\}$, also called \emph{a $4$-book with three pages}.
F\"{u}redi, Pikhurko, and Simonovits \cite{FOS2006} proved that $\EX(n,F_{7})=\{B_{4}(n)\}$ for sufficiently large $n$. 
So $(F_{7},P)$ is a Tur\'an pair, where $P=(\{1,2\},\{\{1,1,2,2\}\})$.
They also proved that $F_{7}$ is degree-stable with respect to  $\Col(P)$.
Hence, by Theorem \ref{PB4}, we obtain the following corollary.

\begin{cor}
For any $\alpha>1$, there exists $n_{0}$ such that if $G$ is a  $F_{7}$-free $4$-graph on $n \ge n_{0}$ vertices, then
$\lambda_{\alpha}(G)\leq\lambda_{\alpha}(B_{4}(n))$, with equality if and only if $G=B_{4}(n)$.
\end{cor}

\section{Proof of Theorem \ref{cri}} \label{Sec-5}
In this section, we establish the proof of Theorem \ref{cri}. We begin by presenting several preliminary lemmas derived under the assumptions and notations of Theorem \ref{cri}.

\begin{fact}\label{fact0}
    For  $p > -1$ and sufficiently large $n$, 
    $$\sum_{i=1}^{n} i^{p} = (1 + o(1)) \frac{n^{p+1}}{p+1}.$$
\end{fact}

\begin{proof}
    The result holds trivially for $p=0$.
    Consider the function $f(x) = x^p$, which is increasing on the interval $[0, n+1]$ for $p>0$, and is decreasing on the interval $[1,n]$ for $-1 < p <0$. 
    
   \textbf{Case 1}: $ p > 0$.  In this case, we have 
    $$\int_{0}^{n} x^p  dx \leq \sum_{i=1}^{n} i^p \leq \int_{1}^{n+1} x^p  dx.$$
    By a direct calculation,
    $$\frac{n^{p+1}}{p+1} \leq \sum_{i=1}^{n} i^p \leq \frac{(n+1)^{p+1} - 1}{p+1},$$
    which implies the desired result.
    
    \textbf{Case 2}: $-1 < p < 0$.  We have 
      $$\int_{1}^{n} x^p  dx + f(n) \leq \sum_{i=1}^{n} i^p \leq \int_{1}^{n} x^p  dx + f(1),$$
    where $ f(1) = 1 $, $ f(n) = n^p $.  
    Therefore,
    $$\frac{n^{p+1} - 1}{p+1} + n^p \leq \sum_{i=1}^{n} i^p \leq \frac{n^{p+1} - 1}{p+1} + 1.$$
    The result follows.
\end{proof}

\begin{fact}\label{fact}
 If $0\leq x<1$ and $\beta>0$, then $(1-x)^{-\beta}\geq 1+\beta x$.
\end{fact}

\begin{fact} \label{fact1}If  $0<x< \frac{1}{2}$, then $1-x\geq e^{-x-x^{2}}$.
\end{fact}

\begin{fact} \label{fact2} If  $x>1$, then $\frac{1}{x}<\ln x-\ln(x-1)$ and $\frac{1}{x^{2}}<\frac{1}{x-1}-\frac{1}{x}$.
\end{fact}

\begin{lem}\label{mind0}
If $n$ is sufficiently large, then 
$\lambda_{\alpha}(\mathcal{G}_{n})\geq \pi(\mathcal{F})(1-2\e')n^{k-k/\alpha}.$  
\end{lem}

\begin{proof}
By (\ref{cri1}) and Fact \ref{fact0},  for sufficiently large  $n\geq N$ we have
\begin{align*}
\begin{split}
\lambda_{\alpha}(\mathcal{G}_{n})
&=\sum_{i=N}^{n}(\lambda_{\alpha}(\mathcal{G}_{i}) -\lambda_{\alpha}(\mathcal{G}_{i-1}))+\lambda_{\alpha}(\mathcal{G}_{N-1})\\
&\geq \sum_{i=N}^{n}
(k-k/\alpha)\pi(\mathcal{F})(1-\e')i^{k-k/\alpha-1}+\lambda_{\alpha}(\mathcal{G}_{N-1})\\
&= \sum_{i=1}^{n}
(k-k/\alpha)\pi(\mathcal{F})(1-\e')i^{k-k/\alpha-1}+O(1)\\
&=(1+o(1))\pi(\mathcal{F})(1-\e')n^{k-k/\alpha}\\
& \geq \pi(\mathcal{F})(1-2\e')n^{k-k/\alpha}.
\end{split}
\end{align*}
\end{proof}

Let $H$ be an $n$-vertex $\mathcal{F}$-free $k$-graph, and  let $\mathbf{x}=(x_{1},\ldots,x_{n})$ be a nonnegative eigenvector corresponding to $\lambda_{\alpha}(H)$.
Note that by definition $\mathbf{x}$ has unit length with respect to $\ell_\alpha$-norm.
Define $x_{\min}=\min\{x_v: v \in V(H)\}$. 
For a subset $U\subseteq V(H)$, denote $x_{U}:=\Pi_{v\in U}x_{v}$.
The following two lemmas mainly discuss the spectral property of $H$.

\begin{lem}\label{mind1}
Let $\e''=\e\pi(\mathcal{F})/(2(k-1))$.
If $\lambda_{\alpha}(H)\geq \lambda_{\alpha}(\mathcal{G}_{n})$ and $\delta(H) < (1-\e)\pi(F)\binom{n}{k-1}$, then for sufficiently large $n$, 
$$x_{\min}^{\alpha}<\frac{1-\e''}{n}.$$
\end{lem}

\begin{proof}
Let  $V:=V(H)$, $\lambda:=\lambda_{\alpha}(H)$, $\delta:=\delta(H)$,  and  let $u\in V$ be a vertex with degree $\delta$. By eigenequation  at  vertex  $u$,
$$\lambda x_{\min}^{\alpha-1}\leq \lambda x^{\alpha-1}_{u}=(k-1)!\sum_{e\in E_{H}(u)}x_{e\backslash\{u\}}.$$
 Applying H\"older's inequality, we have
\begin{equation}\label{e1}
\bigg(\frac{\lambda x_{\min}^{\alpha-1}}{(k-1)!}\bigg)^{\alpha}\leq \delta^{\alpha-1}\sum_{e\in E_{H}(u)}x_{e\backslash\{u\}}^{\alpha}.
\end{equation}
We estimate the right-hand sum above as follows:
\begin{align}\label{e2}
\begin{split}
\sum_{e\in E_{H}(u)}x_{e \backslash\{u\}}^{\alpha}
&=\sum_{S\in \binom{V}{k-1}}x_{S}^{\alpha}-\sum_{T\in \binom{V}{k-1} \text{~and~} T\cup \{u\} \notin E_{H}(u) }x_{T}^{\alpha}\\
&\leq \sum_{S\in \binom{V}{k-1}}x_{S}^{\alpha}-\sum_{T\in \binom{V}{k-1} \text{~and~} T\cup \{u\} \notin E_{H}(u) }x_{\min}^{\alpha(k-1)}\\
&=\sum_{S\in \binom{V}{k-1}}x_{S}^{\alpha}-\left(\tbinom{n}{k-1}-\delta\right)x_{\min}^{\alpha(k-1)}.
\end{split}
\end{align}
By Maclaurin's inequality, we have
\begin{equation}\label{e3}
\sum_{S\in \binom{V}{k-1}}x_{S}^{\alpha}\leq \binom{n}{k-1}\Bigg(\frac{1}{n}\sum_{i\in V}x_{i}^{\alpha}\Bigg)^{k-1}=
\frac{\binom{n}{k-1}}{n^{k-1}}.
\end{equation}

Assume to the contrary that $x_{\min}^{\alpha}\geq\frac{1-\e''}{n}$. Combining (\ref{e2}), (\ref{e3}) and the assumption $\delta(H)< (1-\e)\pi(F)\binom{n}{k-1}$, we have
\begin{align*}
\begin{split}
\sum_{e\in E_{H}(u)}x_{e \backslash\{u\}}^{\alpha}
&\leq\frac{\binom{n}{k-1}}{n^{k-1}}-\big(1-(1-\e)\pi(\mathcal{F})\big)\tbinom{n}{k-1}x_{\min}^{\alpha(k-1)}\\
&\leq\frac{\binom{n}{k-1}}{n^{k-1}}-\big(1-(1-\e)\pi(\mathcal{F})\big)\tbinom{n}{k-1}\frac{(1-\e'')^{k-1}}{n^{k-1}}\\
&=\frac{\binom{n}{k-1}}{n^{k-1}}\Big(1-\big(1-(1-\e)\pi(\mathcal{F})\big)\big(1-\e''\big)^{k-1}\Big).
\end{split}
\end{align*}
 By  Bernoulli's inequality and the definition of $\e''$,   we have
 \begin{displaymath}
\begin{split}
(1-\e'')^{k-1}\geq 1-(k-1)\e'' = 1-\e \pi(\mathcal{F})/2 \geq 1-\e \pi(\mathcal{F}).
\end{split}
\end{displaymath}
Therefore,
\begin{align}\label{e5}
\begin{split}
\sum_{e\in E_{H}(u)}x_{e \backslash\{u\}}^{\alpha}
&\leq \frac{\binom{n}{k-1}}{n^{k-1}}\Big(1-\big(1-(1-\e)\pi(\mathcal{F})\big)\big(1-\e \pi(\mathcal{F})\big)\Big)\\
&= \frac{\binom{n}{k-1}}{n^{k-1}}\Big(\pi(\mathcal{F})-\e(1-\e)\pi(\mathcal{F})^{2}\Big)\\
&\leq\frac{\pi(\mathcal{F})\binom{n}{k-1}}{n^{k-1}}.
\end{split}
\end{align}
Combining this inequality with (\ref{e1}), we have
\begin{align}\label{new11}
\begin{split}
\bigg(\frac{\lambda x_{\min}^{\alpha-1}}{(k-1)!}\bigg)^{\alpha}
&\leq \frac{\pi(\mathcal{F})\binom{n}{k-1}}{n^{k-1}}\bigg((1-\e)\pi(\mathcal{F})\binom{n}{k-1}\bigg)^{\alpha-1}\\
&\leq \frac{(1-\e)^{\alpha-1}\pi(\mathcal{F})^{\alpha}n^{(k-1)(\alpha-1)}}{((k-1)!)^{\alpha}}.
\end{split}
\end{align}

However,  by Lemma \ref{mind0},  for sufficiently large  $n$ we have
$$\lambda\geq\lambda_{\alpha}(\mathcal{G}_{n})
\geq \pi(\mathcal{F})(1-2\e')n^{k-k/\alpha}.$$
This gives the lower bound:
\begin{align}\label{new22}
\begin{split}
\bigg(\frac{\lambda x_{\min}^{\alpha-1}}{(k-1)!}\bigg)^{\alpha}
&\geq \frac{(1-\e'')^{\alpha-1}(1-2\varepsilon')^{\alpha}\pi(\mathcal{F})^{\alpha}n^{(k-1)(\alpha-1)}}{((k-1)!)^{\alpha}}.
\end{split}
\end{align}
Combining (\ref{new11}) and (\ref{new22}) and noting that $\e''\leq\e/2$, 
we have
$$1-2\e'\leq \bigg(\frac{1-\e}{1-\e''}\bigg)^{\frac{\alpha-1}{\alpha}}
\leq\bigg(\frac{1-\e}{1-\e/2}\bigg)^{\frac{\alpha-1}{\alpha}}<\bigg(1-\frac{\e}{2}\bigg)^{\frac{\alpha-1}{\alpha}}.$$
Thus, by Bernoulli's inequality, 
$$2\e'>1-\bigg(1-\frac{\e}{2}\bigg)^{\frac{\alpha-1}{\alpha}}
\ge 1-\bigg(1-\frac{\e (\alpha-1)}{2 \alpha}\bigg)=
\frac{\e(\alpha-1)}{2\alpha}.$$
This contradicts the fact that $$\e'=\frac{\e\pi(\mathcal{F})(\alpha-1)}{2k\alpha}\leq \frac{\e(\alpha-1)}{4\alpha}.$$ This completes the proof of Lemma \ref{mind1}.
\end{proof}

\begin{lem}\label{mind2}
Let  $\e''=\e \pi(\mathcal{F})/(2(k-1))$. If $\lambda_{\alpha}(H)\geq \lambda_{\alpha}(\mathcal{G}_{n})$ and $v$ is a vertex of $H$ such that $x_{v}^{\alpha}<(1-\e'')/n$,  then for sufficiently large $n$,
$$\lambda_{\alpha}(H-v)\geq(1-(k-k/\alpha)(1-\e''/2)n^{-1})\lambda_{\alpha}(H)$$
and
$$\lambda_{\alpha}(H-v)> \lambda_{\alpha}(\mathcal{G}_{n-1}).$$
\end{lem}

\begin{proof}
Let $\mathbf{x}'$ be a sub-vector of $\mathbf{x}$ only by removing the component $x_{v}$. 
For the $k$-graph $H-v$, we have
$$P_{H-v}(\mathbf{x}')=\lambda_{\alpha}(H)-k!x_{v}\sum_{e\in E_{H}(v)}x_{e \backslash\{v\}}=
\lambda_{\alpha}(H)-k\lambda_{\alpha}(H)x_{v}^{\alpha}.$$
Note that 
$P_{H-v}(\mathbf{x}')\leq \lambda_{\alpha}(H-v)(\|\mathbf{x}' \|_{\alpha})^{k}$. 
Since  $x_{v}^{\alpha}<(1-\e'')/n$, by Fact \ref{fact} we have 
\begin{align}\label{e7}
\begin{split}
\frac{\lambda_{\alpha}(H-v)}{\lambda_{\alpha}(H)}&\geq \frac{1-kx_{v}^{\alpha}}{(1-x_{v}^{\alpha})^{k/\alpha}}\\
&\geq (1-kx_{v}^{\alpha})(1+kx_{v}^{\alpha}/\alpha)\\
&=1-(k-k/\alpha)x_{v}^{\alpha}-k^{2}x_{v}^{2\alpha}/\alpha\\
&\geq 1-\frac{(k-k/\alpha)(1-\e'')}{n}
-\frac{k^{2}(1-\e'')^{2}}{\alpha n^{2}},
\end{split}
\end{align}
which implies that for sufficiently large $n$, 
$$\lambda_{\alpha}(H-v)\geq(1-(k-k/\alpha)(1-\e''/2)n^{-1})\lambda_{\alpha}(H).$$

Note that  $H$ is an $n$-vertex $\mathcal{F}$-free $k$-graph with $\lambda_{\alpha}(H)\geq \lambda_{\alpha}(\mathcal{G}_{n})$. 
For sufficiently large $n$, by Theorem \ref{cri}  we have
\begin{equation}\label{replace1}
\lambda_{\alpha}(H)\geq \lambda_{\alpha}(\mathcal{G}_{n}) \ge  \lambda_{\alpha}(\mathcal{G}_{n-1}) + (k-k/\alpha)\pi(\mathcal{F})(1-\e')n^{k-k/\alpha-1}
\end{equation}
and by Lemma \ref{sed},
\begin{equation}\label{replace2}
\lambda_{\alpha}(\mathcal{G}_{n-1})\leq(1+o(1))\pi(\mathcal{F})(n-1)^{k-k/\alpha}\leq(1+o(1))\pi(\mathcal{F})n^{k-k/\alpha}.
\end{equation}
Substituting (\ref{replace1}) and (\ref{replace2}) into (\ref{e7}), we derive
\begin{align*}
\begin{split}
\lambda_{\alpha}(H-v)\geq &\lambda_{\alpha}(\mathcal{G}_{n-1})+(k-k/\alpha)\pi(\mathcal{F})(1-\e')n^{k-k/\alpha-1}\\
&-(k-k/\alpha)\pi(\mathcal{F})(1-\e'')(1+o(1))n^{k-k/\alpha-1}+O(n^{k-k/\alpha-2})\\
\geq&\lambda_{\alpha}(\mathcal{G}_{n-1})+\e(1-1/\alpha)\pi(\mathcal{F})^2 (2(k-1))^{-1} n^{k-k/\alpha-1}+o(n^{k-k/\alpha-1})\\
> &\lambda_{\alpha}(\mathcal{G}_{n-1}),
\end{split}
\end{align*}
where the second equality follows from 
\begin{align*}
(1-\e')-(1-\e'')(1+o(1)) &= \e''- \e' +o(1) \\
& =\e\pi(\mathcal{F})\left( \frac{1}{2(k-1)} - \frac{1}{2k}\left(1-\frac{1}{\alpha}\right)\right) + o(1)\\
& \ge \e\pi(\mathcal{F}) \left( \frac{1}{2(k-1)} - \frac{1}{2k} \right)\\
& = \e \pi(\mathcal{F}) \frac{1}{2k(k-1)}.
\end{align*}
This completes the proof of Lemma \ref{mind2}.
\end{proof}

Finally, we will finish the proof Theorem \ref{cri}.
 
\begin{proof}[{\bf Proof of Theorem \ref{cri}}]
Let $H$ be an $n$-vertex $\mathcal{F}$-free $k$-graph such that $\lambda_{\alpha}(H)\geq\lambda_{\alpha}(\mathcal{G}_{n})$.  
We aim to  show that for sufficiently $n$, $H \in \mathcal{G}_n$, or equivalently, $\delta(H) \ge (1-\e)\pi(F)\binom{n}{k-1}$.
We may assume that $N_{0}$ is sufficiently large to apply Lemmas
 \ref{mind0}, \ref{mind1} and \ref{mind2} for $n\geq N_{0}$. 
 Define 
 $$n_{0}=\bigg(\frac{N_{0}^{k-k/\alpha}\textup{e}^{k^{2}}}{(1-2\e')\pi(\mathcal{F})}\bigg)^{\frac{2}{(k-k/\alpha)\e''}}.$$ 
 Clearly, we have $n_{0}>N_{0}$.  

We assert that when $n \ge n_0$, $H$ has a sub-hypergraph $G$ on $m$ vertices with $\delta(G) \ge (1-\e)\pi(F)\binom{m}{k-1}$ and $m >N_0$.
 The idea is that we can keep removing the vertex of minimum value given by the nonnegative eigenvector associated with the $\alpha$-spectral radius, and then we must eventually get the sub-hypergraph $G$ as wanted.
Suppose this does not give us a suitable sub-hypergraph even after we have got $N_0$ vertices left. 
This means we can  find a sequence of $k$-graphs:
$$H=H_{n} \supseteq  H_{n-1} \supseteq  \cdots \supseteq H_{N_0},$$
where, for each $i>N_0$,  $H_i$ has $i$ vertices with $\delta(H_{i}) < (1-\e)\pi(F)\binom{i}{k-1}$,  and $H_{i-1}=H_i -u_{i}$,  where the vertex $u_{i}\in V(H_{i})$ satisfies $x_{u_{i}}=\min\{x_v: v \in V(H_i)\}$, and $\mathbf{x}$ is a nonnegative eigenvector for $\lambda_{\alpha}(H_{i})$.

By Lemma \ref{mind1}, if $\lambda_{\alpha}(H_{i}) \ge  \lambda_{\alpha}(\mathcal{G}_{i})$ and $\delta(H_{i}) < (1-\e)\pi(F)\binom{i}{k-1}$, then $x_{u_i}^{\alpha}<(1-\e'')/i$; 
 and  by Lemma \ref{mind2}, 
\begin{equation}\label{temp1}
\lambda_{\alpha}(H_{i-1})\geq \lambda_{\alpha}(H_{i})(1-(k-\alpha/k)(1-\e''/2)i^{-1}),
\end{equation}
and
\begin{equation} \label{temp2}
\lambda_{\alpha}(H_{i-1})> \lambda_{\alpha}(\mathcal{G}_{i-1}).
\end{equation}
We note that equation \eqref{temp2} guarantees the repeated application of Lemmas \ref{mind1} and \ref{mind2} such that both equations \eqref{temp1} and \eqref{temp2} hold for all $i > N_0$.

By \eqref{temp1}, we have 
\begin{align*}
\lambda_{\alpha}(H_{N_{0}})
\geq& \lambda_{\alpha}(H_{N_{0}+1})\bigg(1-\frac{(k-k/\alpha)(1-\e''/2)}{N_{0}+1}\bigg)\\
\geq& \lambda_{\alpha}(H_{n})\prod_{i=N_{0}+1}^{n}\bigg(1-\frac{(k-k/\alpha)(1-\e''/2)}{i}\bigg)\\
\geq& \lambda_{\alpha}(H_{n})\exp\Bigg(-\sum_{i=N_{0}+1}^{n}\bigg(\frac{(k-k/\alpha)(1-\e''/2)}{i}+\frac{k^{2}}{i^{2}}\bigg)\Bigg)\\
\geq& \lambda_{\alpha}(H_{n})\exp\bigg(-(k-k/\alpha)(1-\e''/2)\ln\frac{n}{N_{0}}-k^{2}\bigg)\\
\geq& (1-2\e')\pi(\mathcal{F})n^{k-k/\alpha}\Big(\frac{n}{N_{0}}\Big)^{-(k-k/\alpha)(1-\e''/2)}\textup{e}^{-k^{2}}\\
\geq&  (1-2\e')\pi(\mathcal{F})n^{(k-k/\alpha)\e''/2}\textup{e}^{-k^{2}}\\
\geq&  (1-2\e')\pi(\mathcal{F})n_{0}^{(k-k/\alpha)\e''/2}\textup{e}^{-k^{2}}\\
\ge  &  N_{0}^{k-k/\alpha},
\end{align*}
where the 3rd and 4th inequalities follow from Facts \ref{fact1} and \ref{fact2}, respectively.  
This yields a contradiction, since the $\alpha$-spectral radius of any  $k$-graph on $N_{0}$ vertices is less than $N_{0}^{k-k/\alpha}$. 

Hence, the removal process must  terminate at $H_{t}$ for some $t>N_{0}$. By the stopping condition, we have 
$\delta(H_{t}) \ge (1-\e)\pi(\mathcal{F})\binom{t}{k-1},$
which implies  $H_{t}\in \mathcal{G}_{t}$, and hence $\lambda_{\alpha}(H_{t})\leq \lambda_{\alpha}(\mathcal{G}_{t})$. 
If $t<n$, by \eqref{temp2} we have $\lambda_{\alpha}(H_{t})>\lambda_{\alpha}(\mathcal{G}_{t})$, yielding a contradiction. 
Therefore $t=n$, and hence $\delta(H)=\delta(H_{n}) \ge (1-\e)\pi(\mathcal{F})\binom{n}{k-1}$, which implies that $H\in \mathcal{G}_{n}$.
\end{proof}

\section*{Acknowledgments}   We are grateful to Dhruv Mubayi and Xizhi Liu  for helpful comments.

\end{document}